\documentclass[12pt]{amsart}
\usepackage{enumerate}
\theoremstyle{plain}
 \newtheorem{theorem}{Theorem}
 \newtheorem{lemma}{Lemma}
 \newtheorem{corollary}{Corollary}
\newtheorem{proposition}{Proposition}
\theoremstyle{definition}
 \newtheorem{definition}{Definition}
 \newtheorem{remark}{Remark}

 \newcommand{\N}{\mathbb N}

 \newcommand{\R}{\mathbb R}
 
\newcommand{\cA}{\mathcal A}
 
 \newcommand{\cC}{\mathcal C}

 \newcommand{\cF}{\mathcal F}

 \newcommand{\cI}{\mathcal I}
 \newcommand{\cK}{\mathcal K} 
\newcommand{\cL}{\mathcal L}
\newcommand{\cM}{\mathcal M}

\newcommand{\cV}{\mathcal V}

 \newcommand{\xb}{\bar{x}}
 
 \newcommand{\tb}{\bar{t}}

 \newcommand{\af}{\alpha}

 \newcommand{\fui}{\varphi}
 \newcommand{\ep}{\varepsilon}
 \newcommand{\ga}{\gamma}

\newcommand{\dga}{\dot\gamma}
\newcommand{\deta}{\dot\eta}
 
 \newcommand{\de}{\delta}

 \newcommand{\lam}{\lambda}
 \newcommand{\om}{\omega}
 
\newcommand{\si}{\sigma}
 \newcommand{\te}{\theta}

\newcommand{\diam}{\hbox{diam}}
\newcommand{\entre}{\setminus}

\begin{document}
\title[Lax-Oleinik on graphs]{The Lax-Oleinik semigroup on graphs} 
 \author[R. Iturriaga]{Renato Iturriaga}
 \address{ CIMAT, A.P. 402, 3600, Guanajuato. Gto, M\'exico}
 \email{renato@cimat.mx}
 \author[H. S\'anchez-Morgado]{H\'ector S\'anchez-Morgado}
 \address{ Instituto de Matem\'aticas, UNAM. Ciudad Universitaria C. P. 
          04510, Cd. de M\'exico, M\'exico.}
 \email{hector@math.unam.mx}
 \begin{abstract}
We consider Tonelli Lagrangians on a graph, define weak KAM
solutions, which happen to be the fixed points of the Lax-Oleinik
semi-group, and identify their uniqueness set as the Aubry set,
giving a representation formula. 
Our main result is the long time convergence of the Lax Oleinik semi-group.
Weak KAM solutions are viscosity solutions, and in the case of
Hamiltonians called of eikonal type in \cite{CS}, we prove that the converse holds.
 \end{abstract}
 \maketitle 
 \section{Introduction}
 \label{intro}
In the first part of this article we study the Lax-Oleinik semi-group $\cL_t$ defined by a Tonelli Lagrangian on
a graph and prove that for any continuous function $u$, $\cL_tu+c t$ 
converges as $t\to\infty$ where $c$ is the critical value of the Lagrangian.
 
For Lagrangians on compact manifolds, Fathi \cite{F} proved the convergence 
using the Euler-Lagrange flow and conservation of energy. 
In our case we do not have these tools but
we can follow ideas of Roquejoffre \cite{R} and  Davini-Siconolfi \cite{DS}.

Camilli and collaborators \cite{ACCT,CM,CS} have studied viscosity
solutions of the Hamilton-Jacobi equation, and given sufficient conditions
 for a set to be a uniqueness set and a representation formula. 

In the second part of this article we prove that, under the assumption that the
Lagrangian is symmetric at the vertices, the sets of weak KAM and viscosity
solutions of the Hamilton-Jacobi equation defined coincide.

We consider a graph $G$ without boundary 
consisting of finite sets of unoriented edges $\cI=\{I_j\}$ and
vertices  $\cV=\{e_l\}$. The interior of $I_j$ is $I_j-\cV$.
Parametrizing each edge by arc length
$\si_j:I_j\to[0,s_j]$ we can write its tangent bundle as
$TI_j=I_j\times\R$ and
\[TG=\bigcup_i \{j\}\times TI_j/\sim\]
 where $(i,x,v)\sim(j,y,w)\iff (i,x,v)=(j,y,w)$ or 
$x=y\in I_i\cap I_j, v=w=0$.
Thus, a function $L: TG\to\R$ is given by a collection of functions
$L_j: TI_j\to\R$ such that $L_i (e_l,0) = L_j(e_l,0)$ for $e_l\in I_i\cap I_j$.
A Lagrangian in $G$ is a function $L: TG\to\R$ 
such that each $L_j$ is $C^k$, $k\ge 2$, and $L_j(x,\cdot)$ is 
strictly convex and super-linear for any $x\in I_j$. 
We will say that a Lagrangian is {\it symmetric at the vertices} if
at each vertix $e_l$ there is a function $\lam_l:\{u\in\R:u\ge 0\}\to\R$
such that $L_j(e_l,z)=\lam_l(|z|)$ if $e_l\in I_j$. As an
example consider the mechanical Lagrangian given by $L_j(x,v)=\frac 12v^2-U_j(x)$.
For $x\in I_j\entre\cV$, we say that $(x,v)$ points towards 
$\si(s_j)$ if $v>0$ and points towards $\si(0)$ if $v<0$.

We say that $(\si_j^{-1}(0),v)$ is an $I_j$- incoming or outgoing
vector according to whether $v>0$ or $v<0$, and we  say that $(\si_j^{-1}(s_j),v)$ 
is an $I_j$-  incoming or outgoing vector according to wether $v<0$ or $v>0$.
We let $T^+_{e_l}I_j$ ($T^-_{e_l}I_j$) to be the set of $I_j$- outgoing
(incoming or zero) vectors in $T_{e_l}I_j$. 
\section{Basic properties of the action}
\label{sec:basic}

\subsection{A distance on a graph}
\label{sec:distance-graph}
We start defining a distance in the most natural way.
We say a continuous path $\af:[a,b]\to G$ is a 
{\em unit speed geodesic} (u.s.g.) if there is a partition
$a=t_0<\ldots<t_m=b$ such that for each $1\le i\le m$ there is $j(i)$
such that 
\[\af([t_0,t_1])\subset I_{j(1)}, \af([t_1,t_2])=I_{j(2)},\ldots,
\af([t_{m-2},t_{m-1}])=I_{j(m-1)}, \af([t_{m-1},t_m])\subset I_{j(m)},\]
$\si_{j(i)}\circ\af|_{[t_{i-1},t_i]}$ is differentiable and either 
$(\si_{j(i)}\circ\af)'\equiv 1$ or $(\si_{j(i)}\circ\af)'\equiv-1$.
We set the length of a u.s.g. to be
\[\ell(\af)=|\si_{j(1)}(\af(t_1))-\si_{j(1)}(\af(a))|+\sum_{i=2}^{m-1}s_{j(i)}+
|\si_{j(m)}(\af(b))-\si_{j(m)}(\af(t_{m-1}))|\]
and define a distance on $G$ by
\[d(x,y)=\min\{\ell(\af): \af:[a,b]\to G \mbox{ is a u.s.g.},
  \af(a)=x, \af(b)=y\}\]
\subsection{Absolute continuity}
\label{sec:absolute-contiouity}
We say a path $\ga:[a,b]\to G$ is {\em absolutely continuous} if for any
$\ep>0$ there is $\de>0$ such that for any finite collection
of disjoint intervals $\{[c_i,d_i]\}$ with $\sum_i(d_i-c_i)<\de$ we have 
 $\sum_id(\ga(c_i),\ga(d_i))<\ep$.

If $\ga:[a,b]\to G$, $\ga(t)\in\cV$, we define $\dga(t)=0$ if for
any $\ep>0$  there is $\de>0$ such that $d(\ga(s),\ga(t))<\ep|s-t|$
when $|s-t|<\de$.

Let $\ga:[a,b]\to G$ be absolutely continuous and consider 
the closed set $V=\ga^{-1}(\cV)$
so that $(a,b)\entre V=\bigcup_i(a_i,b_i)$ where the intervals $(a_i,b_i)$
are disjoint and $\ga([a_i,b_i])\subset I_{j(i)}$. It is clear that each 
$\si_{j(i)}\circ\ga:[a_i,b_i]\to[0,s_{j(i)}]$ is absolutely continuous.
We set $\dga(t)=(\si_{j(i)}\circ\ga)'(t)$
whenever is defined. 

Next  Proposition will allow us to define the action of an 
absolutely continuous curve.
\begin{proposition}\label{abs-cont}
  Let $\ga:[a,b]\to G$ be absolutely continuous and $V=\ga^{-1}(\cV)$
\begin{enumerate}[(a)]
\item $\dga=0$ Lebesgue almost everywhere in $V$.
\item $\dga$ is integrable and for 
 any $[c,d]\subset[a,b]$ we have
 $d(\ga(c),\ga(d))\le\int\limits_c^d|\dga|$. 
\end{enumerate}
\begin{proof}
Write $(a,b)\entre V=\bigcup_i(a_i,b_i)$ as above with the intervals $(a_i,b_i)$
disjoint.
Since $\dga=0$ on the interior of $V$ and
$\cup\{a_n,b_n\}$ is numerable to stablish item (a) it remains to prove
that  $\dga=0$ Lebesgue almost everywhere in $\partial V\entre\cup\{a_n,b_n\}$.

Let $s=\min_j s_j$ and take $\de>0$ such that $d(\ga(s),\ga(t))<s$ if $|s-t|<\de$
There is $N$ such that $b_i-a_i<\de$ for $i>N$. Since $\ga(a_i), \ga(b_i)\in\cV$    
we have that $\ga(a_i)=\ga(b_i)$ for $i>N$.
We change the labeling of the first $N$ terms
to have $a_1<b_1\le a_2<\cdots<b_m$ and $\ga(a_i)=\ga(b_i)$ for $i>m$.
Letting  $J_0=[a,a_1]$, $J_i=[b_i,a_{i+1}]$,
$1\le i<m$, $J_m=[b_m,b]$, and $V_i=V\cap J_i$ we have $\ga(V_i)=e_{l_i}$,
$0\le i\le m$. We can forget about the cases $b_i=a_{i+1}$.

Define the function $f_i:J_i\to\R$ by $f_i(t)=d(e_{l_i},\ga(t))$.
For $t,s\in J_i$ we have $|f_i(t)-f_i(s)|\le d(\ga(t),\ga(s))$, so $f_i$ is absolutely
continuous and then $f_i'$ exists Lebesgue almost everywhere in $J_i$.
Let $t\in\partial V_i\entre\bigcup_n\{a_n,b_n\}$ be a point where $f_i'$ exists. 
There is a sequence $n_k\to\infty$ such that $a_{n_k}\to t$ 
and $\ga(a_{n_k})=e_{l_i}$. Thus $f_i'(t)=0$, which means that $\dga(t)=0$.

If $(a_j,b_j)\subset J_i$ then $|\dga|=|f_i'|$ Lebesgue 
almost everywhere in $(a_j,b_j)$, so that 
\[\int_{J_i}|\dga|=\int_{J_i}|f_i'|,\]
and then
\[\int_a^b|\dga|=\sum_{i=0}^m\int_{J_i}|\dga|+\sum_{i=1}^m\int_{a_i}^{b_i}|\dga|<\infty.\]
It is also easy to see that for $t,s\in J_i$ we have $d(\ga(t),\ga(s))\le\int\limits_s^t|f_i'|$,
and using the partition $a\le a_1<b_1\le a_2<\cdots<b_m\le b$ to make a
partition of $[c,d]$ we get $d(\ga(c),\ga(d))\le\int\limits_c^d|\dga|$. 

\end{proof}
\end{proposition}

\subsection{Lower semicontinuity and apriori bounds}
\label{sec:lower-semicontunuity}
In this crucial part of the paper we prove that in the
framework of graphs we have the the lower semicontinuity of
the action and apriori bounds for the Lipschitz norm of minimizers.
The proofs have the same spirit as in euclidean space, paying attention
to what happens at the vertices.
Denote by $\cC^{ac}([a,b])$ the set of absolutely
continuous functions $\ga:[a,b]\to G$ provided with the topology of
uniform convergence.  
We define the action of $\ga\in \cC^{ac}([a,b])$ as 
\[A(\ga)=\int_a^bL(\ga(t),\dga(t))dt\]
A minimizer is a $\ga\in \cC^{ac}([a,b])$ such that for any 
$\af\in\cC^{ac}([a,b])$ with $\af(a)=\ga(a)$, $\af(b)=\ga(b)$ we have 
\[A(\ga)\le A(\af)\]
The following two properties of the Lagrangian are important to
achieve our goal and follow from its strict convexity and super-linearity. 
\begin{proposition}\label{Cle}
  If $C\ge 0$, $\ep>0$, there is $\eta>0$ such that for $x,y\in I_j$,  
$d(x,y)<\eta$ and $v,w\in\R$, $|v|\le C$, we have
\[L(y,w)\ge L(x,v)+L_v(x,v)(w-v)-\ep.\]
\end{proposition}
\begin{proposition}\label{Cle2}
If $L_{vv}\ge\te>0$, $C\ge 0$, $\ep>0$, there is $\eta>0$ such that for 
$x,y\in I_j$,  $d(x,y)<\eta$ and $v,w\in\R$, $|v|\le C$, we have
\[L(y,w)\ge L(x,v)+L_v(x,v)(w-v)+\frac{3\te}4|w - v|^2-\ep.\]
\end{proposition}
\begin{lemma}\label{lower}
  Let $L$ be a Lagrangian on $G$. If a sequence $\ga_n\in\cC^{ac}([a,b])$ 
converges uniformly to the curve $\ga:[a,b]\to G$ and 
\[\liminf_{n\to\infty} A(\ga_n)<\infty\]
then the curve $\ga$ is absolutely continuous and 
\[A(\ga)\le\liminf_{n\to\infty} A(\ga_n).\]
\end{lemma}
\begin{proof}
By the super-linearity we may assume that $L\ge 0$.
Let $c= \liminf\limits_{n\to\infty} A(\ga_n)$. Passing to a subsequence we
can assume that 
\[ A(\ga_n)<c+1,\quad \forall n\in\N\]
Fix $\ep>0$ and
take $B> 2(c+1)/\ep$. Again by super-linearity there is a positive number $C(B)$
such that
\[L(x,v)\ge B|v|-C(B),\quad x\in G\entre\cV, v\in \R\]
From Proposition \ref{abs-cont} and $L\ge 0$, for $E\subset[a,b]$ measurable we have 
\[-C(B)\hbox{Leb}(E)+B\int_E|\dga_n|
\le\int_E L(\ga_n,\dga_n)+\int_{[a,b]\entre E}L(\ga_n,\dga_n)
\le c+1.\]
Thus
\[\int_E|\dga_n|\le \frac 1B (c+1+C(B)\hbox{Leb}(E))\le 
\frac\ep 2+\frac{C(B)\hbox{Leb}(E)}B.\]
Choosing $0<\de<\dfrac{\ep B}{2C(B)} $ we have that
\[\hbox{Leb}(E)<\de\Rightarrow\forall n\in\N\, \int_E|\dga_n|<\ep.\]
Since the sequence $\dga_n$ is uniformly integrable, we have that $\ga$ is
absolutely continuous and $\dga_n$
converges to $\dga$ in the $\sigma(L^1,L^\infty)$ weak topology.

Set $V=\ga^{-1}(\cV)$. Let $\ep>0$ and 
$E_k=\bigl\{t:|\dga(t)|\le k, d(t,V)\ge\dfrac 1k\bigr\}$.  

By Propositions \ref{Cle}, \ref{abs-cont},  for $n$ large,
\begin{align*}
  \int_{\ga^{-1}(e_l)}\big[L(e_l,0)+L_v(e_l,0)\dga_n(t)-\ep\big]&\le 
\int_{\ga^{-1}(e_l)} L(\ga_n,\dga_n) ,\\
\int_{E_k}  \big[L(\ga,\dga)+L_v(\ga,\dga)(\dga_n-\dga)-\ep\big]
  &\le\int_{E_k} L(\ga_n,\dga_n), \\
\int_{E_k\cup V} \big[L(\ga,\dga)+L_v(\ga,\dga)(\dga_n-\dga)-\ep\big]
 & \le\int_{E_k\cup V} L(\ga_n,\dga_n)
\le A(\ga_n) .
\end{align*}
  Letting $n\to+\infty$ we have that 
  \[  \int_{E_k\cup V}L(\ga,\dga)\le c+\ep\;(b-a)\]
  Since $E_k\uparrow [a,b]\entre V$ when $k\to+\infty$ and $L\ge 0$, we have
  \[ A(\ga)=\lim_{k\to+\infty}\int_{E_k\cup V}L(\ga,\dga)\le c+\ep\;(b-a)\]
  
  Now let $\ep\to 0$. 
\end{proof}
Lemma \ref{lower} implies
\begin{theorem}\label{semicontinua} 
Let $L$ be a Lagrangian on $G$.
The action $A:\cC^{ac}([a,b])\to\R\cup\{\infty\}$ is lower semicontinuous. 
\end{theorem}

\begin{proposition}\label{tonelli}
Let $L$ be a Lagrangian on $G$.

The set $\{\ga\in\cC^{ab}([a,b]): A(\ga)\le K\}$
is compact with the topology of uniform convergence.
\end{proposition}
Let $C_t=\sup\{L(x,v):x\in G, |v|\le\frac{\diam(G)}t\}$,
then for any minimizer $\ga:[a,b]\to G$ with $b-a\ge t$ we have 
\[A(\ga)\le C(b-a).\]
\begin{proposition}\label{addendum}
 Suppose $\ga_n:[a,b]\to G$ converge uniformly to $\ga:[a,b]\to G$ and 
 $A(\ga_n)$ converges to $A(\ga)$, then $\dot\ga_n$ converges to 
 $\dot\ga$ in $L^1[a,b]$ 
\end{proposition}
\begin{proof}
Let $F\subset[a,b]$ be a finite union of intervals.
From Lemma \ref{lower} we have 
\[\int_FL(\ga,\dga)\le\liminf_n\int_F L(\ga,\dga) \hbox{ and }
\int_{[a,b]\entre F}L(\ga,\dga)\le\liminf_n\int_{[a,b]\entre F}L(\ga,\dga)\]
Since 
\[\lim_n\int_FL(\ga_n,\dga_n)+\int_{[a,b]\entre F}L(\ga_n,\dga_n)
=\lim_nA(\ga_n)=A(\ga)\]
we have 
\begin{equation}
  \label{eq:I}
\lim_n\int_{F}L(\ga_n,\dga_n)=\int_{F}L(\ga,\dga).
\end{equation}
If $V=\ga^{-1}(\cV)\subset F$ then 
\[\limsup_n\int_{ V}L(\ga_n,\dga_n)\le\lim_n\int_{F}L(\ga_n,\dga_n)
=\int_{F}L(\ga,\dga).\]   
Thus 
\begin{equation}
  \label{eq:0}
  \limsup_n\int_{ V}L(\ga_n,\dga_n)\le \int_{ V}L(\ga,\dga).
\end{equation}
As in Lemma \ref{lower}, the $\dga_n$ are uniformly integrable, so they 
converge to $\dga$ in the $\sigma(L_1,L_\infty)$ weak topology and then,
for any Borel set $B$ where $\dga$ is bounded,
\begin{equation}
  \label{eq:II}
\lim_n\int_{B}L_v(\ga,\dga)(\dga_n-\dga)=0.  
\end{equation}
Given $\ep>0$, from 
Proposition \ref{Cle2} we have that for $n$ large enough 
\[ \frac{3\te}4\int_{\ga^{-1}(e_l)}|\dga_n|^2\le\int_{\ga^{-1}(e_l)}
\big[L(\ga_n,\dga_n)-L(e_l,0)-L_v(e_l,0)\dga_n+\ep\big].
\]
Which together with Proposition \ref{abs-cont} and equations \eqref{eq:0}, \eqref{eq:II} give 
$\limsup\limits_n\int\limits_{ V}|\dga_n|^2\le\ep$ for any $\ep>0$. So 
\begin{equation}
  \label{eq:a}
  \lim_n\int_{ V}|\dga_n|^2=0.
\end{equation}
For $k>0$ let 
$D_k:=\{t\in[a,b]:|\dga(t)|>k\}$, $B_k:=\{t\in[a,b]: d(t,V)>\frac 1k\}$.
Then $\lim\limits_{k\to\infty}$Leb$(D_k)=\lim\limits_{k\to\infty}$Leb$([a,b]\entre V\entre B_k)=0$.
Let $F_k$ be a finite union of intervals such that
$D_k\cap B_k\subset F_k\subset B_k$ and Leb$(F_k\entre(D_k\cap B_k))<\frac 1k.$
Then $\lim\limits_{k\to\infty}$ Leb$(F_k)=0$.

Given $\ep>0$, from 
Proposition \ref{Cle2} we have that for $n$ large enough 
\[  \frac{3\te}4\int_{B_k\entre F_k}|\dga_n-\dga|^2\le\int_{B_k\entre F_k}
\big[L(\ga_n,\dga_n)-L(\ga,\dga)-L_v(\ga,\dga)(\dga_n-\dga)+\ep\big].
\]
From \eqref{eq:I}, \eqref{eq:II} we get that
$\limsup\limits_n\int\limits_{B_k\entre F_k}|\dga_n-\dga|^2\le\ep$
for any $\ep>0$. So
\begin{equation}
  \label{eq:b}
  \lim_n\int_{B_k\entre F_k}|\dga_n-\dga|^2=0.
\end{equation}
Since $\{\dga_n\}$ is uniformly integrable, given $\ep>0$, for $k$ 
sufficiently large we have
\begin{equation}\label{eq:c}
  \int_{F_k\cup[a,b]\entre V\entre B_k}|\dga_n-\dga|\le
\int_{F_k\cup[a,b]\entre V\entre B_k}|\dga_n|+|\dga|<\ep, 
\end{equation}
From \eqref{eq:a}, \eqref{eq:b}, \eqref{eq:c}
and Cauchy-Schwartz inequality, we have that for any $\ep>0$
\[ \limsup_n\int_a^b|\dga_n-\dga|\le
\lim_n\underset{V}{\int}|\dga_n|+
\limsup_n\underset{F_k\cup[a,b]\entre V\entre B_k}{\int|\dga_n-\dga|}+
\lim_n\underset{B_k\entre F_k}{\int}|\dga_n-\dga|\le \ep\]

\end{proof}
\begin{lemma}\label{aprioriac}
Let $L$ be a Lagrangian in $G$. 
 For $\ep>0$ there exists $K_{\ep}$ that is a Lipschitz
 constant for any minimizer $\ga:[a,b]\to G$ with 
$b-a\ge \ep$.
\end{lemma}
\begin{proof}
Note that if $\ga$ is a minimizer and $\ga(c,d)\subset I_j$ then 
$\ga|_{(c,d)}$ is a solution of the Euler Lagrange equation for $L_j$.

Suppose the Lemma is not true, then by Proposition \ref{abs-cont},
for any $i\in\N$ there are a minimizer  
$\ga_i: [s_i,t_i]\to G$ with $t_i-s_i\ge\ep$ and a set
$E_i\subset [s_i,t_i]\entre\ga_i^{-1}(\cV)$ with
Leb$(E_i)>0$ such that $|\dga|>i$ on $E_i$. Let $c_i\in E_i$.
Translating $[s_i,t_i]$ we can assume that $c_i=c$ for all $i$ and taking 
a subsequence that there is $a\in\R$ such that $\ga_i$ is defined in 
$[a,a+\frac{\ep}2], c\in[a,a+\frac{\ep}2]$. 
As $A(\ga_i|[a,a+\frac{\ep}2])$ is bounded, by Proposition \ref{tonelli} there 
is subsequence $\ga_i|[a,a+\frac{\ep}2]$ which converges uniformly to 
$\ga:[a,a+\frac{\ep}2]\to G$. Since $\ga$ is limit of minimizers, it is a 
minimizer and $A(\ga)\le\liminf A(\ga_i|[a,a+\frac{\ep}2])$. 
We can not have that $A(\ga)<\limsup A(\ga_i|[a,a+\frac{\ep}2])$
because that would contradict that the $\ga_i$ are minimizers. Thus 
$A(\ga)=\lim A(\ga_i|[a,a+\frac{\ep}2])$. 

If $\ga(c)\in I_j\entre\cV$, there is $\de>0$ such that 
$\ga([c-\de,c+\de])\subset I_j$. 

If $\ga(c)=e_l$ we have 2 possibilities (not mutually exclusive) 

a) There is an edge  $I_j$ with $e_l\in I_j$ and infinitely many $i$'s 
such that  $\ga_i(c)\in I_j$ and $\dot\ga_i(c)$ points towards $e_l$. 

b) There is an edge $I_j$ with $e_l\in I_j$ and infinitely many $i$'s such that 
$\ga_i(c)\in I_j$ and $\dot\ga_i(c)$ points towards the other vertex. 



In case a) there is $\de>0$ such that $\ga([c-\de,c])\subset I_j$. 

In case b) there is $\de>0$ such that $\ga([c,c+\de])\subset I_j$. 

We have that $\ga$ is a solution of the Euler-Lagrange equation for 
$L_j$ either on $[c-\de,c]$ or on $[c,c+\de]$ and then $|\dot\ga(t)|\le K$ on 
$[c-\de,c]$ or $[c,c+\de]$. 
For some $0<\de_1<\de$ we have that $\ga_i$ are solutions 
of the Euler-Lagrange equation for $L_j$ on $[c-\de_1,c]$ or on $[c,c+\de_1]$. 
For $i$ suficiently large, we have that $|\dot\ga_i|>2 K$ either
on $[c-\de_1,c]$ or on $[c,c+\de_1]$. This would contradict
Proposition \ref{addendum}.
\end{proof}

\section{Weak KAM theory on graphs}
\label{wkam-graph}
The content of this section is similar to that for Lagrangians
on compact manifolds. We only give the proofs that are different from
those in the compact manifold case, which can be found in \cite{F}.
\subsection{The Peierls barrier}
\label{sec:barrier}
Given $x, y \in G$ let $\cC^{ac}(x,y,t)$ be the set of curves 
$\alpha\in\cC^{ac}([0,t])$ such that $\alpha (0)=x$ and $\alpha (t) = y$.
For a given real number $k$ define 
$$h_t(x,y)=\min_{\alpha \in \cC^{ac}(x,y,t)} A(\alpha) $$
 and
$$h^k(x,y)= \liminf_{t\rightarrow\infty}h_t(x,y)+kt $$
\begin{lemma}\label{hLip}
For $\ep>0$ the function $F:[\ep,\infty)\times G\times G\to\R$
defined by  $F(t,x,y)= h_t(x,y)$ is Lipschitz.
 \end{lemma}
 
\begin{lemma}\label{critical}
There exists a real $c$ independent of $x$ and $y$ such that
\begin{enumerate}
\item For all $k>c$ we have $h^k(x,y) =\infty $.
\item  For all $k<c$ we have $h^k(x,y) =-\infty $
\item $h^c(x,y)  $ is finite. The function $h:=h^c$ is called the \emph{Peierls barrier}.
\end{enumerate}
\end{lemma}

\begin{lemma}\label{c-cerradas}
The value $c$ is the infimum of $k$ such that $\int\limits_\gamma L+k\ge 0$
for all closed curves $\ga$. 
\end{lemma}

\begin{definition}\quad
 The \emph{Ma\~n\'e potencial} $\Phi:G\times G\to\R$ is defined by
\[\Phi(x,y)=\inf_{t>0} h_t(x,y)+ct.\]
Clearly  we have $\Phi(x,y)\le h(x,y)$ for any $x,y\in G$.
\end{definition}
\begin{proposition}\label{propiedades-h}
  Functions $h$ and $\Phi$ have the following properties.
  \begin{enumerate}
  \item \label{3fi} $\Phi(x,z)\le \Phi(x,y)+\Phi(y,z)$.  
\item \label{3h} $h(x,z)\le h(x,y)+\Phi(y,z)$, $h(x,z)\le\Phi(x,y)+h(y,z)$. 
\item \label{lip} $h$ and  $\Phi$ are Lipschitz  
  \item \label{unif} If $\ga_n:[0,t_n]\to G$ is a sequence of absolutely continuous curves
with $t_n\to\infty$ and $\ga_n(0)\to x$, $\ga_n(t_n)\to y$, then
\begin{equation}
  \label{ineq-unif}
  h(x,y)\le\liminf_{n\to\infty}A(\ga_n)+ct_n.
\end{equation}
\end{enumerate}
\end{proposition}

\begin{definition}
A curve $\ga:J\to G$ is called 
       \begin{itemize}
       \item {\em semi-static} if
\[\Phi(\ga(t),\ga(s)=\int_t^s L(\ga,\dga)+c(s-t)\]
for any $t,s\in J$, $t\le s$.
\item {\em static} if
\[\int_t^s L(\ga,\dga)+c(s-t)=-\Phi(\ga(s),\ga(t))\]
for any $t,s\in J$, $t\le s$.
\item 
The \emph{Aubry set} $\cA$ is the set of points $x\in G$ such that  
$h(x,x)=0$.
\end{itemize}
    \end{definition}

Notice that by item \eqref{3h} in Proposition \ref{propiedades-h}, 
$h(x,z)=\Phi(x,z)$ if $x\in\cA$ or $z\in\cA$.
\begin{proposition}
If $\eta:\R\to G$ is static then $\eta(s)\in\cA$ for any $s\in\R$.  
\end{proposition}
Although we do not conservation of energy we can prove
that semi-static curves have energy $c(L)$.
\begin{proposition}\label{energy}
  Let $\eta:J\to G$ be semi-static. For almost every $t\in J$
\[L_v(\eta(t),\dot\eta(t))\dot\eta(t) = L(\eta(t),\dot\eta(t))+c\]
 \end{proposition}
  \begin{proof}
 For $\lam>0$, let $\eta_\lam(t):=\eta(\lam t)$ so that 
$\dot\eta_\lam(t)=\lam\dot\eta(\lam t)$ almost everywhere.

For $r,s\in J$ let
$$\cA_{rs}(\lam):=\int_{r/\lam}^{s/\lam}[L(\eta_\lam(t),\dot\eta_\lam(t))+c]\, dt 
=\int_r^s[L(\eta(s),\lam\dot\eta(s))+c]\,\frac{ds}{\lam}. $$
Since $\eta$ is a free-time minimizer, differentiating 
$\cA_{rs}(\lam)$ at $\lam=1$, we have that
  $$  0=\cA_{rs}'(1)=
    \int_0^T[L_v(\eta(s),\dot\eta(s))\dot\eta(s)-L(\dot\eta(s),\dot\eta(s)-c]\,ds.
  $$
Since this holds for any $r,s\in J$ we have 
\[L_v(\eta(t),\dot\eta(t))\dot\eta(t) = L(\eta(t),\dot\eta(t))+c\]
for almost every $t\in J$. 
  \end{proof}
 
\subsection{Weak KAM solutions}\label{wkam-sol}
Following Fathi \cite{F}, we define weak KAM solutions and give some of
their properties
\begin{definition}
Let $c$ be given by Lemma \ref{critical}.
\begin{itemize}
\item A function $u:G\to\R$ is \emph{dominated} if for any $x,y\in G$,  we have 
\[u(y)-u(x)\le h_t(x,y)+ct\quad \forall t>0 ,\]
or equivalently 
\[u(y)-u(x)\le\Phi(x,y).\] 
\item 
$\ga:I\to G$ \emph{calibrates} a dominated function $u:G\to\R$ if
\[u(\ga(s))-u(\ga(t))=\int_t^sL(\ga,\dot\ga)+c(s-t)\quad\forall s,t\in I\] 
\item 
A continuous function $u:G\to\R$ is a 
\emph{backward (forward) weak KAM solution} if
it is dominated and for any $x\in G$ there is $\ga:(-\infty,0]\to G$
($\ga:[0,\infty)\to G$) that calibrates $u$ and  $\ga(0)=x$ 
\end{itemize}
\end{definition}
\begin{corollary}\label{est-calibra}
 Any static curve $\ga:J\to G$ calibrates any dominated function
 $u:G\to\R$ 
\end{corollary}

\begin{proposition}\label{h-wkam}
For any $x\in G$, $h(x,\cdot)$ is a backward weak KAM solution and 
$-h(\cdot,x)$ is a forward weak KAM solution. 
\end{proposition}
\begin{proof}
By item \eqref{3h} of Proposition \ref{propiedades-h}, $h(x,\cdot)$ is dominated.

The standard construction of calibrating curves  for compact manifolds
involves the Euler Lagange flow that we do not have, so
we use a diagonal trick. 
Let $\ga_n:[-t_n,0]\to G$ be a sequence of minimizing curves 
connecting $x$ to $y$ such that 
\[h(x,y)=\lim_{n\to\infty}A(\ga_n)+ct_n\]
By Lemma \ref{aprioriac}, $\{\ga_n\}$ is uniformly Lipschitz and then
equicontinuous. 
It follows from the Arzela Ascoli Theorem that there is a sequence 
$n^1_j\to\infty$ such that  $\ga_{n^1_j}$ converges uniformly on $[-1, 0]$.
Again, by the Arzela Ascoli Theorem, there is a subsequence $(n^2_j)_j$ of 
the sequence $(n^1_j)_j$ such that $\ga_{n^1_j}$ converges uniformly on $[-2,0]$.
By induction, this procedure gives for each $k\in\N$ a sequence $(n^k_j)_j$ that
is a subsequence of the sequence $(n^{k-1}_j)_j$ and such that $\ga_{n^k_j}$ 
converges uniformly on $[-k,0]$ as $j\to\infty$. Letting $m_k=n^k_k$, the sequence
$\ga_{m_k}$ converges  uniformly on each $[-l,0]$. For $s<0$ define 
$\ga(s)=\lim\limits_{k\to\infty}\ga_{m_k}(s)$.
Fix $t<0$, for $k$ large $t+t_{m_k}\ge 0$ and
\begin{equation}\label{diag}
  A(\ga_{m_k})+ct_{m_k}=
\int\limits_{-t_{m_k}}^tL(\ga_{m_k},\dot\ga_{m_k})+c(t+t_{m_k})+
\int_t^0L(\ga_{m_k},\dot\ga_{m_k})-ct.
\end{equation}
Since  $\ga_{m_k}$ converges to $\ga$ uniformly on  $[t,0]$, we have 
\[\liminf_{k\to\infty}\int_t^0L(\ga_{m_k},\dot\ga_{m_k})\ge \int_t^0L(\ga,\dot\ga).\]
From item \eqref{unif} of Proposition \ref{propiedades-h} we have
\[h(x,\ga(t))\le\liminf_{k\to\infty}
\int_{-t_{m_k}}^tL(\ga_{m_k},\dot\ga_{m_k})+c(t+t_{m_k}).\]

Taking $\liminf\limits_{k\to\infty}$ in \eqref{diag} we get
\[h(x,y)\ge h(x,\ga(t))+\int_t^0L(\ga,\dot\ga)-ct.\]
So $\ga$ calibrates $h(x,\cdot)$.
\end{proof}

From Proposition \ref{h-wkam} we have 
\begin{corollary}\label{aubry-calibra}
  If $x\in\cA$ there exists a curve $\ga:\R\to G$ such that $\ga(0)=x$ and for 
all $t\ge 0$
  \begin{align*}
    h(\ga(t),x)&=-\int_0^tL(\ga,\dot\ga)-ct\\
    h(x,\ga(-t))&=-\int_{-t}^0L(\ga,\dot\ga)-ct.\\
  \end{align*}
In particular the curve $\ga$ is static and calibrates any dominated function 
$u:G\to\R$. 
\end{corollary}
\begin{theorem}\label{aubry-potential-sol}
  The function $\Phi(x,\cdot)$ is a backward weak KAM solution if and
  only if $x\in\cA$.
\end{theorem}

\begin{corollary}\label{kam}
  Let $C\subset G$ and $w_0:C\to\R$ be bounded from below. Let
\[w(x)=\inf_{z\in C}w_0(z)+\Phi(z,x)\]
\begin{enumerate}
\item\label{max-dom} 
$w$ is the maximal dominated function not exceeding $w_0$ on $C$.
\item\label{aubry-kam}
If $C\subset\cA$, $w$ is a backward weak KAM solution. 
\item\label{dom-coinc} If for all $x,y\in C$
\[w_0(y)-w_0(x)\le\Phi(x,y),\]
then $w$ coincides with $w_0$ on $C$.
\end{enumerate}
\end{corollary}
For $u:G\to\R$ let $I(u)$ be the set of points $x\in G$ for which exists 
$\ga:\R\to G$ such that $\ga(0)=x$ and $\ga$ calibrates $u$.
\begin{corollary}\label{calibra-todo}
\[\cA=\bigcap\limits_{u \emph{ dominated}}I(u)\]
  \end{corollary}

  \begin{proposition}
For each $x,y\in G$ with $x\ne y$ we can find $\ep>0$ and a curve 
$\ga:[-\ep,0]\to G$ such that $\ga(0)=y$ and for all $t\in[0,\ep]$
\[\Phi(x,\ga(0))-\Phi(x,\ga(-t))=\int_{-t}^0L(\ga,\dot\ga)+ct.\]
In particular, for each $x\in G$ the function $G\entre\{x\}\to\R$; 
$y\mapsto\Phi(x,y)$ is a backward weak KAM solution.
  \end{proposition}

\begin{theorem}\label{gonzalo}
$\cA$ is nonempty and 
if $u:G\to\R$ is a backward weak KAM solution then
\begin{equation}\label{RF}
  u(x)=\min_{q\in\cA}u(q)+h(q,x)
\end{equation}
\end{theorem}

\begin{corollary}\label{gonza}
\[h(x,y)=\min_{q\in\cA}h(x,q)+h(q,y)=\min_{q\in\cA}\Phi(x,q)+\Phi(q,x)\] 
\end{corollary}

  \section{The Lax semigroup and its convergence}
 \subsection{The Lax semigroup}
Let $\cF$ be the set of real functions on $G$, bounded from below. 

The backward Lax semigroup $\cL_t:\cF\to\cF$, $t>0$ is defined by
\[\cL_t f (x)= \inf_{y\in G}f(y)+h_t(y,x).\]
It is clear that $f\in \cF$ is dominated if and only if $f\le\cL_tf+ct$ 
for any $t>0$.

It follows at once that $\cL_t\circ\cL_s=\cL_{t+s}$ and 
\begin{equation}
  \label{eq:cd}
  \|\cL_t f-\cL_t g\|_\infty\le\|f-g\|_\infty
\end{equation}

 The proof of the following Lemma is the same as in the compact
manifold case.
\begin{lemma}
 Given $\ep>0$ there is $K_\ep>0$ such that for each $u:G\to\R$ continuous,
$t\ge\ep$, we have $\cL_tu:G\to\R$ is a Lipschitz with constant $K_\ep$.
\end{lemma}

\begin{theorem}\label{fixed=weak}
A continuous function $u:G\to\R$ is a fixed point of the semigroup $\cL_t+ct$
if and only if it is a backward weak KAM solution 
\end{theorem}
\begin{proof}
Suppose $u:G\to\R$ is a fixed point of the semigroup $\cL_t+ct$.
For each $T\ge 2$ there is a curve $\af_T: [-T, 0]\to G$ such 
that $\af_T(0)=x$ and  
\[u (x)-u (\af_T (-T))=A (\af_T) + cT. \] 
By Lemma \ref{aprioriac} $\{\af_T\}$ is uniformly Lipschitz. 
As in Propostion \ref{h-wkam}
one obtains a sequence $t_k\to\infty $ and
$\ga:(-\infty,0]\to G$ such that  $\af_{t_k}$ converges to 
$\ga$, uniformly on each $[-n, 0]$. 

By Lemma \ref {semicontinua} 
\begin{align*}
\int_{-n}^0L(\ga,\dot\ga)+nc &\le\liminf_{k\to\infty} \int_{-n}^0L(\af_{t_k},\dot\af_{t_k})+nc\\
&=\liminf_{k\to\infty} u(x)-u(\af_{t_k}(-n))\\
  &=u(x)-u(\ga(-n))
\end{align*}
Suppose now that $u:G\to\R$ is a backward weak KAM solution. 
Since $u$ is dominated, $u\le\cL_t u+ct$.
For $x\in G$ let $\ga:(-\infty,0]\to G$ be such that $\ga(0)=x$ and for all $t>0$
\[u(x)-u(\ga(-t))=\int_{-t}^0L(\ga,\dot\ga)+ct.\]
Thus
\[u(x)\ge u(\ga(-t))+h_t(\ga(t),x)+ct\ge\cL_tu(x)+ct.\]
\end {proof} 
From Proposition \ref{h-wkam} and Theorem \ref{fixed=weak} one obtains
\begin{corollary}
The semigroup $\cL_t+ct$ has fixed points.
\end{corollary}
\subsection{Convergence of the Lax semigroup}
\label{sec:lax-converge}
Without loss of generality assume $c=0$. For $u\in C(G)$ define
\begin{equation}
  \label{eq:limite}
  v(x):=\min_{z\in G}u(z)+h(z,x). 
\end{equation}
\begin{proposition}\label{ge}
Let $\psi=\lim\limits_{n\to\infty}\cL_{t_n}u$ for some $t_n\to\infty$, then  
\begin{equation}\label{arriba}
\psi\ge v.
\end{equation}
\end{proposition}
\begin{proof}
  For $x\in G$ let $\ga_n:[0,t_n]\to G$ be such that $\ga_n(t_n)=x$ and
\begin{equation}\label{other}
\cL_{t_n}u(x)=u(\ga_n(0))+A(\ga_n).  
\end{equation}
Passing to a subsequence if necessary we may assume that 
$\ga_n(0)$ converges to $y\in G$. Taking $\liminf$ in \eqref{other},
we have from item \eqref{unif} of Proposition \ref{propiedades-h} 
\[\psi(x)=u(y)+\liminf_{n\to\infty}A(\ga_n)\ge u(y)+h(y,x). \]
\end{proof}
\begin{proposition}\label{igual}
If $\cL_tu$ converges as $t\to\infty$, then the limit is 
function $v$ defined in \eqref{eq:limite}.   
\end{proposition}
\begin{proof}
  For $x\in G$ let $z\in G$ be such that $v(z)=u(z)+h(z,x)$. 
Since $\cL_t u(x)\le u(z)+h_t(z,x)$,  we have  
\[\lim_{t\to\infty}\cL_t u(x)\le \liminf_{t\to\infty}u(z)+h_t(z,x)= v(z)\] 
which together with Proposition \eqref{ge} gives $\lim\limits_{t\to\infty}\cL_t u=v$. 
\end{proof}  
Thus, given $u\in C(G)$ our goal is to prove that
$\cL_tu$ converges to $v$ defined in \eqref{eq:limite}.  
\begin{remark}
  Using Corollary \ref{gonza} we can write \eqref{eq:limite} as
\begin{align}\label{eq:limite1}
  v(x)&=\min_{y\in\cA}\Phi(y,x)+w(y)\\
  \label{eq:w}
  w(y)&:= \inf_{z\in G}u(z)+\Phi(z,y)
\end{align} 
Item \eqref{max-dom} of Corollary \ref{kam} states that $w$
is the maximal dominated function not exceeding $u$.
Items \eqref{aubry-kam}, \eqref{dom-coinc} of the same Corollary imply that
$v$ is the unique backward weak KAM solution that coincides with $w$
on $\cA$.
\end{remark}
\begin{proposition}\label{dom-conv}
 Suppose that $u$ is dominated, then  $\cL_tu$ converges uniformly as 
$t\to\infty$ to the function $v$ given by \eqref{eq:limite}.
\end{proposition}
\begin{proof}
Since $u$ is dominated, the function $t\mapsto\cL_tu$ nondecreasing. 
As well, in this case, $w$ given by \eqref{eq:w} coincides with $u$. 
Items \eqref{max-dom} and \eqref{dom-coinc} of
Corollary \ref{kam} imply that $v$ is the maximal dominated function
that coincides with $u$ on $\cA$ and then $u\le v$ on $G$.

Since the semigroup $\cL_t$ is monotone and $v$ is a backward weak KAM solution
\[\cL_tu\le\cL_tv=v\hbox{ for any } t>0.\] 
Thus the uniform limit $\lim\limits_{t\to\infty}u$ exists.
\end{proof}
We now address the convergence of $\cL_t$ following the lines in
\cite{DS} which coincide in part  with those in \cite{R}. 

For $u\in C(G)$ let
\[\om_\cL(u):=\{\psi\in C(G):\exists t_n\to\infty \hbox{ such that } 
\psi=\lim_{n\to\infty}\cL_{t_n}u\}.\]
\begin{align}\label{uu}
  \underline{u}(x)&:=\sup\{\psi(x):\psi\in\om_\cL(u)\}\\
  \overline{u}(x)&:=\inf\{\psi(x):\psi\in\om_\cL(u)\}\label{ou}
\end{align}
From these and Proposition \ref{ge}
\begin{proposition}\label{v<us}
  Let $u\in C(G)$, $v$ be the function given by \eqref{eq:limite},
  $\underline{u}, \overline{u}$ defined in \eqref{uu} and \eqref{ou}. Then 
  \begin{equation}\label{eq:v<u}
    v\le \overline{u}\le\underline{u}
  \end{equation}
\end{proposition}
  
\begin{proposition}
For $u\in C(G)$,  function $\underline{u}$
given by   \eqref{uu} is dominated.
\end{proposition}
\begin{proof}
Let $x,y\in G$. Given $\ep>0$ there is $\psi=\lim\limits_{n\to\infty}\cL_{t_n}u$ 
such that $\underline{u}(x)-\ep<\psi(x).$ For $n>N(\ep)$ and $a>0$
\[\underline{u}(x)-2\ep<\psi(x)-\ep\le \cL_{t_n}u(x)=\cL_a(\cL_{t_n-a}u)(x)\le
\cL_{t_n-a}u(y)+h_a(y,x).\]
Choose a divergent sequence  $n_j$ such that $(\cL_{t_{n_j}-a}u)_j$
converges uniformly. For $j>\bar N(\ep)$, $\cL_{t_{n_j}-a}u(y)<\underline{u}(y)+\ep$,
and then
\[\underline{u}(x)-3\ep<\cL_{t_{n_j}-a}u(y)+h_a(y,x)-\ep<\underline{u}(y) +h_a(y,x).\]
\end{proof}

Denote by $\cK$ the family of static curves $\eta:\R \to G$, 
and for $y\in\cA$ denote  by $\cK(y)$ the set of curves $\eta\in\cK$ 
with $\eta(0)=y$.
\begin{proposition}
$\cK$ is a compact metric space with respect to the uniform convergence on 
compact intervals.  
\end{proposition}
\begin{proof}
Let $\{\eta_n\}$ be  a sequence in $\cK$. By Lemma \ref{aprioriac}, $\{\eta_n\}$ is uniformly Lipschitz.
As in Proposition \ref{h-wkam} we obtain a sequence $n_k\to\infty$ such that
$\eta_{n_k}$ converges to $\eta:\R\to G$ uniformly on each $[a,b]$ and
then $\eta$ is static.
\end{proof}

\begin{proposition}\label{Acoincide}
Two dominated functions that coincide on 
$\cM=\bigcup\limits_{\eta\in\cK}\om(\eta)$ also coincide on $\cA$.
 \end{proposition}
 \begin{proof}
  Let $\fui_1$, $\fui_2$ be two dominated functions coinciding on
  $\cM$. Let $y\in\cA$ and $\eta\in\cK(y)$. Let $(t_n)_n$ be a
  diverging sequence such that $\lim_n\eta(t_n)= x\in\cM$. 
By Corollary \ref{est-calibra} 
\[\fui_i(y) = \fui_i(\eta(0)) - \Phi(y, \eta(0)) = \fui_i(\eta(t_n)) -
\Phi(y, \eta(t_n))\] 
for every $n\in N, i=1,2$. Sending $n$ to $\infty$, we get
\begin{align*}
\fui_1(y)&= \lim_{n\to\infty}\fui_1(\eta(t_n)) -\Phi(y, \eta(t_n))
=\fui_1(x) - \Phi(y, x) =\fui_2(x) - \Phi(y, x) \\
&= \lim_{n\to\infty}\fui_2(\eta(t_n)) -\Phi(y, \eta(t_n))
=\fui_2(y).
\end{align*}
 \end{proof}
  \begin{proposition}\label{nonincreasing}
Let $\eta\in\cK$, $\psi\in C(G)$ and $\fui$ be a dominated function. 
Then the function
$t\mapsto(\cL_t\psi)(\eta(t))-\fui(\eta(t))$ is nonincreasing on $\R_+$.
 \end{proposition}
 \begin{proof} From Corollary \ref{est-calibra}, for $t<s$ we have 
   \[(\cL_s\psi)(\eta(s))-(\cL_t\psi)(\eta(t))\le\int_t^s L(\eta(\tau),\dot\eta(\tau))d\tau 
=\fui(\eta(s))-\fui(\eta(t))\]
 \end{proof}

 \begin{lemma}\label{perturbacion}
   There is a $M>0$ such that, if $\eta$ is any curve in $\cK$ and $\lam$ 
is sufficiently close to 1, we have
   \begin{equation}\label{ineq}
\int_{t_1}^{t_2}L(\eta_\lam,\dot{\eta}_\lam)\le
\Phi(\eta_\lam(t_1),\eta_\lam(t_2))+M(t_2-t_1)(\lam-1)^2
   \end{equation}
for any $t_2>t_1$, where  $\eta_\lam(t)=\eta(\lam t)$.
\end{lemma}
\begin{proof}
Let $K>0$ be a Lipschitz constant for any minimizer $\ga:[a,b]\to G$ with $b-a>1$,
 $2R=\sup\{|L_{vv}(x,v)|:|v|\le K\}$. 
For $\lam\in(1-\de,1+\de)$ fixed, using Proposition \ref{energy} 

\begin{align*}
\int_{t_1}^{t_2}L(\eta_\lam(t),\dot\eta_\lam(t))dt
&=\int_{t_1}^{t_2}[L(\eta(\lam t),\dot\eta(\lam t))+(\lam-1)
L_v(\eta(\lam t),\dot\eta(\lam t)) \dot\eta(\lam t)\\
&+\frac 12(\lam-1) ^2L_{vv}(\eta(\lam t),\mu\dot\eta(\lam t))(\dot\eta(\lam t))^2]\, dt\\
&\le\lam \int_{t_1}^{t_2}L(\eta(\lam t),\dot\eta(\lam t))\,dt+(t_2-t_1)RK^2(\lam-1)^2\\
&=\Phi(\eta(\lam t_1),\eta(\lam t_2))+(t_2-t_1)RK^2(\lam-1)^2 
 \end{align*}
\end{proof}
\begin{proposition}\label{superdiff}
   Let $\eta\in\cK$, $\psi\in C(G)$ and $\fui$ be a dominated function. 
Assume that $D^+((\psi-\fui)\circ\eta)(0)\entre\{0\}\ne\emptyset$
where $D^+$ denote the super-differential.
Then for all $t>0$ we have
\begin{equation}\label{eq:2}
(\cL_t\psi)(\eta(t))-\fui(\eta(t))<\psi(\eta(0))-\fui(\eta(0))
\end{equation}
 \end{proposition}
 \begin{proof}
   Fix $t>0$. By Corollary \ref{est-calibra} it is enough to prove
   \eqref{eq:2} for $\fui=-\Phi(\cdot,\eta(t))$. Since $\cL_t(\psi+a)=\cL_t\psi+a$
we can assume that $\psi(\eta(0))=\fui(\eta(0))$.
\[(\cL_t\psi)(\eta(t))-\fui(\eta(t))=(\cL_t\psi)(\eta(t))\le
\int_{(1/\lam-1)t}^{t/\lam}L(\eta_\lam,\dot{\eta}_\lam)+\psi(\eta((1-\lam)t)),\]
thus, by Lemma \ref{perturbacion}
 \[(\cL_t\psi)(\eta(t))-\fui(\eta(t))\le
\psi(\eta((1-\lam)t))-\fui(\eta((1-\lam)t))+Mt(\lam-1)^2.\]
If $m\in D^+((\psi-\fui)\circ\eta)(0)\entre\{0\}$, we have
 \[(\cL_t\psi)(\eta(t))-\fui(\eta(t))\le m((1-\lam)t)+o((1-\lam)t))+Mt(\lam-1)^2,\]
where $\lim\limits_{\lam\to 1}\dfrac{o((1-\lam)t)}{1-\lam}=0$. 
Choosing appropriately $\lam$
close to $1$, we get
\[(\cL_t\psi)(\eta(t))-\fui(\eta(t))<0.\]
\end{proof}
\begin{proposition}\label{superincreasing}
  Suppose $\fui$ is dominated and $\psi\in\om_\cL(u)$. For any 
$y\in\cM$ there exists $\ga\in\cK(y)$ such that the function
$t\mapsto\psi(\ga(t))-\fui(\ga(t))$ is constant. 
\end{proposition}
\begin{proof}
 Let $(s_k)_k$ and $(t_k)_k$ be diverging sequences, $\eta$ be a
 curve in $\cK$ such that $y=\lim\limits_k \eta(s_k)$, and $\psi$ is
 the  uniform limit of $\cL_{t_k}u$. As in Proposition \ref{h-wkam}, 
we can assume that the sequence of functions $t\mapsto\eta(s_k+t )$ converges 
uniformly on compact intervals to $\ga:\R\to G$, and so  $\ga\in\cK$. 
We may assume moreover that $t_k-s_k\to\infty$, as $k\to\infty$, and
that $\cL_{t_k-s_k}u$ converges uniformly to $\psi_1\in\om_\cL(u)$. 
By the semi-group property and \eqref{eq:cd} 
\[\|\cL_{t_k}u-\cL_{s_k}\psi_1\|_\infty\le\|\cL_{t_k-s_k}u -\psi_1\|_\infty\]
which implies that $\cL_{s_k}\psi_1$ converges uniformly to $\psi$. 
From Proposition \ref{nonincreasing}, we have that for any $\tau\in\R$
$s\mapsto(\cL_s\psi_1)(\eta(\tau+s))-\fui(\eta(\tau+s))$ is a nonincreasing 
function in $\R^+$, and hence it has a limit $l(\tau)$ as $s\to\infty$, 
which is finite since $l(\tau)\ge-\|\overline u-\fui\|_\infty$. 
Given $t>0$, we have
\[l(\tau) = \lim_{k\to\infty} (\cL_{s_k+t}\psi_1) (\eta(s_k +\tau+t))-\fui(\eta(s_k+\tau+t)) = (\cL_t\psi) (\ga(\tau+t))-\fui(\ga(\tau+t))\]
The function $t\mapsto(\cL_t\psi)(\ga(\tau+t))-\fui(\ga(\tau+t))$ is therefore constant on $\R^+$. 
Applying Proposition \ref{superdiff} to
the curve $\ga(\tau +\cdot)\in\cK$, we have
$D^+((\psi-\fui)\circ\ga)(\tau)\entre\{0\} =\emptyset$ for any
$\tau\in\R$. This implies that $\psi-\fui$ is constant on $\ga$.
\end{proof}
\begin{proposition}\label{acercanse}
  Let $\eta\in\cK$, $\psi\in\om_\cL(u)$ and $v$ be defined by \eqref{eq:limite}.
For any $\ep>0$ there exists $\tau\in\R$ such that 
\[\psi(\eta(\tau))-v(\eta(\tau))<\ep.\]
\end{proposition}
\begin{proof}
  Since the curve $\eta $ is contained in $\cA$, we have
\[v(\eta(0)) = \min_{z\in G} u(z) + \Phi(z,\eta(0)),\] 
and hence $v(\eta(0)) = u(z_0) + \Phi(z_0, \eta(0))$, for some
$z_0\in G$. Take a curve $\ga:[0,T]\to G$ such that
\[v(\eta(0))+\frac\ep 2 = u(z_0)+\Phi(z_0,\eta(0))+\frac\ep 2
> u(z_0)+\int_0^TL(\ga,\dga)\ge\cL_Tu(\eta(0)).\]
Choosing a divergent sequence $(t_n)_n$ such that $\cL_{t_n}u$
converges uniformly to $\psi$ we have for $n$ sufficiently large
\[\|\cL_{t_n}u-\psi\|_\infty<\frac\ep 2, \quad t_n-T>0.\]
Take $\tau=t_n-T$ 
\begin{align*}
  \psi(\eta(\tau))-\frac\ep 2&<\cL_{t_n}u(\eta(\tau)=\cL_{\tau}\cL_Tu\\
&\cL_Tu(\eta(0)+\int_0^\tau L(\eta,\dot\eta)\\
&\frac\ep 2+v(\eta(0))+\int_0^\tau L(\eta,\dot\eta)=\frac\ep 2+v(\eta(\tau))
\end{align*}
\end{proof}
From Propositions \ref{superincreasing} and \ref{acercanse} we obtain
\begin{theorem}\label{Mcoincide}
 Let $\psi\in\om_\cL(u)$ and $v$ be defined by \eqref{eq:limite}. Then $\psi=v$ 
on $\cM$.
\end{theorem}
\begin{theorem}
  Let $u\in C(G)$, then $\cL_tu$ converges uniformly as $t\to\infty$ to $v$ 
given by \eqref{eq:limite}.
\end{theorem}
\begin{proof}
 The function $\underline{u}$
is dominated and coincides with $v$ on $\cM$ by Theorem 
\ref{Mcoincide}. Proposition \ref{Acoincide} implies that
$\underline{u}$ coincide with $v$ on $\cA$ and so does with $w$. 
By item \eqref{max-dom} of Corollary \ref{kam} we have  $\underline{u}\le v$.
\end{proof}
\section{Viscosity solutions of the Hamilton - Jacobi equation}
\label{sec:visco-sol}
In this section we compare weak KAM and viscosity solutions.
\begin{definition}\quad
  \begin{itemize}
  \item A real function $\fui$ defined on the neighborhood of $e_l$ is $C^1$ if for 
every $j$ with $e_l\in I_j$, $\fui|I_j$ is $C^1$.
\item A real function $\fui$ defined on the neighborhood of $(e_l,t) $ is $C^1$ 
if for every $j$ with $e_l\in I_j$, $\fui|I_j\times(t-\de,t+\de)$ is $C^1$.
  \end{itemize}
\end{definition}
Note that if $\af: [0,\de]\to I_j$ is differentiable and $\af(0)=e_l$, then
$\af_+'(0)\in T^-_{e_l}I_j$ and
we have 
\[D^j\fui(e_l)z=(\fui\circ \af)_+'(0).\]

We consider the Hamiltonian consisting in 
functions $H_j:I_j\times\R\to\R$ given by
\[H_j(x,p)=\max\left\{-pz-L_j(x,z) :
\begin{array}{ll}z\in T^-_xI_j, & x\in\cV \\  
z\in T_xI_j, & x\in I_j\entre\cV
\end{array}\right\}
\]
and the Hamilton Jacobi equations
\begin{equation}
  \label{eq:HJ}
  H(x,Du(x))=c,
\end{equation}
\begin{equation}\label{eq:hjt}
  u_t(x,t)+H(x,D_xu(x,t))=0.
\end{equation}

Note that if $L$ is symmetric at the vertices, then for any vertix $e_l$ there
is a function $h_a$ such that $H_j(e_l,p)=h_a(|p|)$ for any $j$ with 
$e_l\in I_j$. This kind of Hamiltonians are called of eikonal type \cite{CS}.

The following definition appeared in \cite{CS} and \cite{CM}.
\begin{definition}
A function $u:G\to\R$ is a
\begin{itemize}
\item {\em viscosity subsolution} of \eqref{eq:HJ} if
satisfies the usual definition in $G\entre\cV$ and for any $C^1$
function $\fui$ on the neighborhood of any $e_l$ s.t. $u-\fui$ 
has a maximum at $e_l$ we have
\[\max\{H_j(e_l,D^j\fui(e_l)):e_l\in I_j\}\le c.\] 
\item {\em viscosity supersolution} of \eqref{eq:HJ} if
satisfies the usual definition in $G\entre\cV$ and for any $C^1$
function $\fui$ on the neighborhood of any $e_l$ s.t. $u-\fui$ 
has a minimum at $e_l$ we have
\[\max\{H_j(e_l,D^j\fui(e_l)):e_l\in I_j\}\ge c\]
\item  {\em viscosity solution} if it is both, a subsolution and a 
supersolution.
\end{itemize}
A function $u:G\times[0,\infty)\to\R$ is a
\begin{itemize}
\item {\em viscosity subsolution} of \eqref{eq:hjt} if
satisfies the usual definition in $G\entre\cV\times[0,\infty)$ and for any $C^1$
function $\fui$ on the neighborhood of any $(e_l,t)$ s.t. $u-\fui$ 
has a maximum at $(e_l,t)$ we have
\[\fui_t(e_l,t)+\max\{H_j(e_l,D^j\fui(e_l,t)):e_l\in I_j\}\le c.\] 
\item {\em viscosity supersolution} of \eqref{eq:HJ} if
satisfies the usual definition in $G\entre\cV\times[0,\infty)$ and for any $C^1$
function $\fui$ on the neighborhood of any $(e_l,t)$ s.t. $u-\fui$ 
has a minimum at $(e_l,t)$ we have
\[\fui_t(e_l,t)+\max\{H_j(e_l,D^j\fui(e_l)):e_l\in I_j\}\ge c\]
\item {\em viscosity solution} if it is both, a subsolution and a supersolution.
\end{itemize}
\end{definition}

\begin{proposition}\label{KAM=viscosity}
If $u:G\to\R$ is dominated then then it is a viscosity subsolution
of \eqref{eq:HJ}. If $u$ is a backward weak KAM solution then it is a 
viscosity solution.
\end{proposition}
\begin{proof} Suppose $u:G\to\R$ is dominated.
Let $\fui$ be a $C^1$ function on the neighborhood of $e_l$ s.t. $u-\fui$ 
has a maximum at $e_l$, $j$ s.t. $e_l\in I_j$, $\af:[0,\de]\to I_j$ differentiable with
$\af(0)=e_l$, $z=\af'(0)$. Define $\ga:[-\de,0]\to I_j$ by $\ga(s)=\af(-s)$.
\begin{align*}
  \fui(e_l)-\fui(\ga(s))&\le  u(e_l)-u(\ga(s))\le\int_s^0L_j(\ga,\dot\ga)-cs\\
\frac{\fui(e_l)-\fui(\af(t)))}t&\le\frac 1t\int_{-t}^0L_j(\ga,\dot\ga)+c\\
-D^j\fui(e_l)z&\le L_j(e_l,z)+c.
\end{align*}
So $u$ is a subsolution.

 Let $\fui$ be a $C^1$ function on the neighborhood of $e_l$ s.t. $u-\fui$ has a minimum 
at $e_l$. Let $\ga:(-\infty,0]\to G$ be such that $\ga(0)=e_l$ and for $t<0$
\[u(e_l)-u(\ga(t))=\int_t^0L_j(\ga,\dot\ga)-ct\]
Let $\de>0, j$ be such that $\ga([-\de,0])\subset I_j$.
\[\fui(e_l)-\fui(\ga(s))\ge\int_s^0L_j(\ga,\dot\ga)-cs\]
Define $\af:[0,\de]\to I_j$ by $\af(t)=\ga(-t)$, $z=\af'(0)$, 
\begin{align*}
  \frac{\fui(e_l)-\fui(\af(t))}t&\ge\frac 1t\int_{-t}^0L_j(\ga,\dot\ga)+c\\
-D^j\fui(e_l)z&\ge L_j(e_l,z)+c.
\end{align*}
So $u$ is a supersolution.
\end{proof}
\begin{proposition}\label{laxsemi}
 Let $f:G\to\R$ be continuous and define $u:G\times[0,\infty)\to\R$ by 
$u(x,t)=\cL_tf(x)$, then $u$ is a viscosity solution of \eqref{eq:hjt}
\end{proposition}
\begin{proof}
Since $\cL_tf=\cL_{t-s}(\cL_sf)$ if $0\le s<t$, for any $\ga:[s,t]\to G$
\begin{equation}
  \label{domina}
u(\ga(t),t)-u(\ga(s),s)\le\int_s^tL(\ga,\dot\ga)
\end{equation}
and for any $x\in G$ there is $\ga:[s,t]\to G$ with $\ga(t)=x$ such
that equality in \eqref{domina} holds.

Let $\fui$ be a $C^1$ function on the neighborhood of $(e_l,t)$ s.t. $u-\fui$ 
has a maximum at $(e_l,t)$, $j$ s.t. $e_l\in I_j$, $\af:[0,\de]\to I_j$
differentiable with 
$\af(0)=e_l$, $z=\af'(0)$. Define $\ga:[t-\de,t]\to I_j$ by
$\ga(s)=\af(t-s)$.
\begin{align*}
\fui(e_l,t)-\fui(\ga(s),s)&\le  
u(e_l,t)-u(\ga(s),s)\le\int_s^tL_j(\ga,\dot\ga)\\
\frac{\fui(e_l,t)-\fui(\af(t-s),s))}{t-s}&
\le\frac 1{t-s}\int_{s}^tL_j(\ga,\dot\ga)\\ 
\fui_t(e_l,t)-D^j_x\fui(e_l,t)z &\le L_j(e_l,z).
\end{align*}
So $u$ is subsolution.

Let $\fui$ be a $C^1$ function on the neighborhood of $(e_l,t)$ s.t. $u-\fui$ 
has a minimum at $(e_l,t)$. Let $\ga:[t-1,t]\to G$ be such that $\ga(t)=e_l$ and
\[u(e_l,t)-u(\ga(t-1),t-1)=\int_{t-1}^tL(\ga,\dot\ga)\]
Let $\de>0, j$ be such that $\ga([t-\de,t])\subset I_j$. For $s\in[t-\de,t]$
\[\fui(e_l,t)-\fui(\ga(s),s)\ge\int_s^tL_j(\ga,\dot\ga)\]
Define $\af:[0,\de]\to I_j$ by $\af(s)=\ga(t-s)$, $z=\af'(0)$, 
\begin{align*}
\frac{\fui(e_l,t)-\fui(\af(t-s),s)}{t-s}&\ge\frac 1{t-s}\int_s^tL_j(\ga,\dot\ga)\\
\fui_t(e_l,t)-D^j_x\fui(e_l,t)z&\ge L_j(e_l,z).
\end{align*} 
So $u$ is supersolution.
\end{proof}

\begin{proposition}\label{comparacion}
Suppose the Lagrangian is symmetric at the vertices.
Let $u,v:G\times[0,T]\to\R$ be respectively a Lipschitz viscosity sub, supersolution of  
\eqref{eq:hjt} such that $u(x,0)\le v(x,0)$, for any  $x\in G$. Then $u\le v$.
\end{proposition}
\begin{proof}
  Suppose that there are $x^*,t^*$ such that $\de=u(x^*,t^*)- v(x^*,t^*)>0$.
Let $0<\rho\le\dfrac\de{4t^*}$ and define $\Phi:G^2\times[0,T]^2$ by
 \[\Phi(x,y,t,s)=u(x,t)-v(y,s)-\frac{d(x,y)^2+|t-s|^2}{2\ep}-\rho(t+s),\]
where $d(x,y)$ is the shortest lenght of a path in $G$ connecting $x$
and $y$, and so $d(x,y)=d(y,x)$

From the  previous definitions we have
\begin{equation}
  \label{eq:3}
\frac\de 2\le\de-2\rho t^*=\Phi(x^*,x^*,t^*,t^*)\le\sup_{G^2\times[0,T]^2}\Phi=
\Phi(x_\ep,y_\ep,t_\ep,s_\ep). 
\end{equation}
It follows from 
$\Phi(x_\ep,x_\ep,t_\ep,t_\ep)+\Phi(y_\ep,y_\ep,s_\ep,s_\ep)\le 2\Phi(x_\ep,y_\ep,t_\ep,s_\ep)$ 
that
\begin{align*}
  \frac{d(x_\ep,y_\ep)^2+|t_\ep-s_\ep|^2}{2\ep}&\le
u(x_\ep,t_\ep)-u(y_\ep,s_\ep)+v(x_\ep,t_\ep)-v(y_\ep,s_\ep)\\
&\le C(d(x_\ep,y_\ep)^2+|t_\ep-s_\ep|^2)^{1/2}
\end{align*}
Thus, there is a sequence $\ep\to 0$ such that $x_\ep,y_\ep$
converge to $\xb\in G$ and $t_\ep,s_\ep$ converge to
$\tb\in[0,T]$ and \eqref{eq:3} gives
\[\frac\de 2\le\Phi(\xb,\xb,\tb,\tb)\le u(\xb,\tb)-v(\xb,\tb),\]
and so $\tb\ne 0$. 
Define the test functions 
\begin{align*}
  \fui(x,t)&=v(y_\ep,s_\ep)+\frac{d(x,y_\ep)^2+|t-s_\ep|^2}{2\ep}+\rho(t+s_\ep)\\
\psi(y,s)&=u(x_\ep,t_\ep) -\frac{d(x_\ep,y)^2+|t_\ep-s|^2}{2\ep}-\rho(t_\ep+s).
\end{align*}
\[\fui_t (x_\ep,t_\ep)=\dfrac{t_\ep-s_\ep}\ep+\rho, 
\quad\psi_s(y_\ep,s_\ep)=\dfrac{t_\ep-s_\ep}\ep-\rho\]
Since $u-\fui$ has maximum at $(x_\ep,t_\ep)$, $v-\psi$ has minimum
at $(y_\ep,s_\ep)$, $u$ is subsolution and $v$ is supersolution,
\begin{align}\nonumber
  2\rho=\fui_t (x_\ep,t_\ep)-\psi_s(y_\ep,s_\ep) \le
& \max\{H_j\bigl(y_\ep,-D^j_y\Bigl(\frac{d(x_\ep,y)^2}{2\ep}\Bigr)(y_\ep)\bigr):
x\in I_j\}\\
- & \max\{H_j\bigl(x_\ep,D^j_x\Bigl(\frac{d(x,y_\ep)^2}{2\ep}\Bigr)(x_\ep)\bigr):
x\in I_j\}
\label{fundamental}
\end{align}
Since $\rho>0$ we can not have $x_\ep=y_\ep$. 


If $\xb$ is not a vertix, $\xb\in I_j$, for $\ep>0$ small we have
\[D^j_x\Bigl(\frac{d(x,y_\ep)^2}{2\ep}\Bigr)(x_\ep)=
\pm\frac{d(x_\ep,y_\ep)}{\ep}=-D^j_y\Bigl(\frac{d(x_\ep,y)^2}{2\ep}\Bigr)(y_\ep).\]
If we denote by $a(x_\ep,y_\ep)$ this common value, then \eqref{fundamental} becomes
\[2\rho\le H_j(y_\ep,a(x_\ep,y_\ep))
-  H_j\bigl(x_\ep,a(x_\ep,y_\ep))\]
with $a(x_\ep,y_\ep)$ bounded as $\ep\to 0$, 
giving a contradiction.

Suppose now that $\xb=e_l$. For $\ep>0$ small we
distinguish the following cases

1. Neither $x_\ep$ nor $ y_\ep$ is a vertix. 
If $x_\ep, y_\ep\in I_j$, $d(x_\ep,y_\ep)=|\si_j(x_\ep)-\si_j(y_\ep)|$.  
If $x_\ep\in I_i$, $ y_\ep\in I_j$, and 
$e_l\in I_i\cap I_j$, then  $d(x_\ep,y_\ep)=d(x_\ep,e_l)+d(e_l,y_\ep)$.
In both subcases
\begin{align*}
|D^i_x\Bigl(\frac{d(x,y_\ep)^2}{2\ep}\Bigr)(x_\ep)|&=
\frac{d(x_\ep,y_\ep)}{\ep}\\
|D^j_y\Bigl(\frac{d(x_\ep,y)^2}{2\ep}\Bigr)(y_\ep)|&=
\frac{d(x_\ep,y_\ep)}{\ep}
\end{align*}
Then \eqref{fundamental} becomes 
\[2\rho\le H_j\bigl(y_\ep,\pm\frac{d(x_\ep,y_\ep)}{\ep}\bigr)
-  H_i\bigl(x_\ep,\pm\frac{d(x_\ep,y_\ep)}{\ep}\bigr).\]

2. Suppose $x_\ep=e_l$, $y_\ep\in I_j\entre\cV$.
\begin{align*}
|D^j_x\Bigl(\frac{d(x,y_\ep)^2}{2\ep}\Bigr)(e_l)|&=
\pm\frac{d(e_l,y_\ep)}{\ep}\\
|D^j_y\Bigl(\frac{d(e_l,y)^2}{2\ep}\Bigr)(y_\ep)|&=
\pm\frac{d(e_l,y_\ep)}{\ep}
\end{align*}
Since 
\[H_j\bigl(e_l,\pm\frac{d(e_l,y_\ep)}{\ep}\bigr)
=h_a\bigl(e_l,\frac{d(e_l,y_\ep)}{\ep}\bigr),\]
we have that \eqref{fundamental} becomes
\[2\rho\le H_j\bigl(y_\ep,\pm\frac{d(x_\ep,y_\ep)}{\ep}\bigr)
-H_j\bigl(x_\ep,\frac{d(x_\ep,y_\ep)}{\ep}\bigr).\]

3. If $y_\ep=e_l$, $x_\ep\in I_j\entre\cV$ we get in the same way that 
\eqref{fundamental} becomes
\[2\rho\le H_j\bigl(y_\ep,\frac{d(x_\ep,y_\ep)}{\ep}\bigr)
-H_j\bigl(x_\ep,\pm\frac{d(x_\ep,y_\ep)}{\ep}\bigr).\]

Since $\dfrac{d(x_\ep,y_\ep)}{\ep}$ remains bounded as $\ep\to 0$
in all cases, we get a contradiction.
\end{proof}
\begin{corollary}\label{unicidad}
Suppose the Lagrangian is symmetric at the vertices.
  Let $u,v:G\times[0,T]\to\R$ be viscosity solutions of \eqref{eq:hjt}
such that $u(x,0)= v(x,0)$ for any $x\in G$. Then $u=v$.
\end{corollary}
\begin{corollary}\label{viscosity=fixed}
Suppose the Lagrangian is symmetric at the vertices.
 Let $f:G\to\R$ be a viscosity solution of \eqref{eq:HJ}, then 
$f$ is a fixed point of the Lax semigroup $\cL_t+ct$.
\end{corollary}
\begin{proof}
  We next show that $u(x,t)=f(x)-ct$ is a viscosity solution of \eqref{eq:hjt}.
Proposition \ref{laxsemi} and Corollary \ref{unicidad} then imply that
$f-ct=\cL_tf$. 

Let $\fui$ be a $C^1$ function on the neighborhood of $(e_l,t)$ s.t. $u-\fui$ 
has a maximum at $(e_l,t)$. Then $s\to -cs-\fui(e_l,s)$ has a maximum at
$t$ and so $\fui_t(e_l,t)=-c$. Since $f-\fui(\cdot,t)$ has a maximum
at $e_l$ we have 
\[\sup\{H_j(x,D^j\fui(x)):x\in I_j\}\le c=-\fui_t(e_l,t),\]
so $u$ is a subsolution of \eqref{eq:hjt}. Similarly $u$ is a
supersolution of \eqref{eq:hjt}.
\end{proof}
\begin{corollary}\label{con-unic}
Suppose the Lagrangian is symmetric at the vertices.
Let $u:G\to\R$ be a viscosity solution of \eqref{eq:HJ} then 
the representation formula \eqref{RF} holds.
\end{corollary}
\begin{proof}
  By Proposition \ref{KAM=viscosity} and Corollary \ref{viscosity=fixed}, 
$u$ is a backward weak KAM solution and by Theorem \eqref{gonzalo},
formula \eqref{RF} holds.
\end{proof}

\end{document}